\documentclass[12pt,leqno]{amsart}
\usepackage[dvips]{graphics}
\usepackage{amssymb}

\numberwithin{equation}{section}

\newtheorem{thm}{THEOREM}[section]

\newtheorem{lem}[thm]{Lemma}
\newtheorem{cor}[thm]{Corollary}

 \theoremstyle{definition}

\theoremstyle{remark}
\newtheorem{rem}{Remark}[section]

\newcommand{\RP}{\mathbb{R\mkern1mu P}}

\newcommand{\tref}[1]{Theorem~\ref{#1}}
\newcommand{\cref}[1]{Corollary~\ref{#1}}

\newcommand{\lref}[1]{Lemma~\ref{#1}}

\newcommand{\pr}{\mathrm{pr}}
\newcommand{\folL}{\mathcal{L}}
\newcommand{\he}{\hat{e}}\newcommand{\tY}{\tilde{Y}}\newcommand{\tS}{\tilde{S}}
\newcommand{\gS}{\mathsf{S}}
\newcommand{\gQ}{\mathsf{Q}}
\newcommand{\gT}{\mathsf{T}}
\newcommand{\tX}{\tilde{X}}\newcommand{\hX}{\hat{X}}\newcommand{\Aut}{\mathrm{Aut}}

\newcommand{\hB}{\hat{B}}\newcommand{\bg}{\bar{g}}
\newcommand{\fF}{\mathbb{F}}\newcommand{\hM}{\hat{M}}
\newcommand{\id}{\mathrm{id}}
\newcommand{\fN}{\mathbb{N}}\newcommand{\fQ}{\mathbb{Q}}
\newcommand{\R}{\mathbb{R}}\newcommand{\C}{\mathbb{C}}\newcommand{\Z}{\mathbb{Z}}\newcommand{\Spin}{\mathsf{Spin}}
\newcommand{\SO}{\mathsf{SO}}\newcommand{\Sph}{\mathbb{S}}
\newcommand{\Or}{\mathsf{O}}\newcommand{\gN}{\mathsf{N}}

\begin{document}
\pagebreak


\title{Riemannian foliations of spheres}

\author{Alexander Lytchak}
\thanks{The first named author was partially supported by a Heisenberg grant of the
DFG and both authors by the SFB~``Groups, Geometry and Actions''}
\address{Mathematisches Institut\\ Universit\"at zu K\"oln\\
Weyertal 86-90, 50931 K\"oln, Germany\\}
\email{alytchak\@@math.uni-koeln.de}
\author{Burkhard Wilking}
\address{Mathematisches Institut\\ Universit\"at M\"unster\\
Einsteinstr. 62, 48147 M\"unster, Germany\\}
\email{wilking\@@uni-muenster.de}
\subjclass{53C12,  57R30}




\begin{abstract}
We show that a Riemannian foliation on a topological $n$-sphere has leaf 
dimension $1$ or $3$ unless $n=15$ and the Riemannian foliation is given by the fibers 
of a Riemannian submersion to an $8$-dimensional sphere.
This allows us to 
classify Riemannian foliations on round spheres up to metric congruence.
\end{abstract}

\maketitle
\renewcommand{\theequation}{\arabic{section}.\arabic{equation}}
\pagenumbering{arabic}


\section{Introduction}

We are going to prove the final piece of the following theorem:
\begin{thm}\label{cor: main}\label{thm: main} Suppose $\mathcal F$ is a Riemannian foliation by $k$-dimensional leaves
of a compact manifold $(M,g)$  which is homeomorphic to $\Sph^n$. 
We assume $0<k<n$.
Then one of the following holds
\begin{enumerate}
\item[a)] $k=1$  and the foliation is given by an isometric flow, with respect to some 
Riemannian metric. 
\item[b)] $k=3$, $n\equiv 3\,\mathrm{ mod }\,4$ and the  generic leaves are diffeomorphic to 
$\RP^3$ or $\Sph^3$.
\item[c)] $k=7$, $n=15$ and $\mathcal F$ is given by the fibers of a Riemannian submersion $(M,g)\rightarrow (B,\bar g)$
where $(B,\bg)$ is homeomorphic to $\Sph^8$ and the fiber is homeomorphic to $\Sph^7$.
\end{enumerate}
Furthermore all these cases can occur.
\end{thm}

A big part of the Theorem follows by putting  together various pieces in the literature: 
Ghys (\cite{ghys}) showed that the generic leaves of a Riemannian foliation of a homotopy sphere are closed,
unless the leave dimension is $1$ and the foliation is given by an isometric flow with respect to 
a possibly different Riemannian metric.
Furthermore, the generic leaves are rational homotopy  spheres. Haefliger (\cite{classify}) observed
that for any  Riemannian foliation of a complete manifold $M$ with 
closed leaves, one can find a space $\hat M$ homotopically equivalent to $M$ such 
that $\hat M$ is the total space of a fiber  bundle, where the fibers are homeomorphic to the generic
leaves of the foliation (see section~\ref{sec: top} for further details). If $M$ is a sphere then the
fibers are contractible in $\hat M$. Spanier and Whitehead  observed (\cite{spanier}) that for any such
fibration the fiber must be 
an $H$-space. Furthermore, closed manifolds which are  $H$-spaces and rational homotopy spheres were classified
by  Browder (\cite{Browder}): they are homotopically equivalent to $\Sph^1$, $\RP^3$, $\Sph^3$, $\RP^7$ or $\Sph^7$.
With Perelman's solution of the geometrization conjecture 
one can improve the statement further to diffeomorphic if $k=3$.
 
 We are left to consider $7$-dimensional foliations of homotopy spheres. Our strategy will be 
 reduce the situation first to the case of $n=15$ and then show that the foliation is simple, i.e. given 
 by the fibers of a Riemannian submersion. 
 By a result of Browder  (\cite{Browder})  this automatically rules out the possibility 
 of an $\RP^7$-foliation.
 
To see that all examples can occur, we can again appeal to the literature
for the  only non-classical case: the existence of $\RP^3$ foliations on $\Sph^{4k+3}$.
 It was shown by Oliver (\cite{Oliver}), that contrary to a previous conjecture, there are almost free 
  smooth 
actions of $\SO(3)\cong \RP^3$ on $\Sph^{4k+3}$ for $k\ge 1$. 
The actions of Oliver extend to fixed point free smooth actions on the disc $D^{4k+4}$, 
different actions were later exhibited  by Grove and Ziller (\cite{groveziller}). 
 
 Our topological result allows us to classify Riemannian foliations of the round sphere up to metric congruence. 
We recall 
that Gromoll and Grove (\cite{gromollsph} classified Riemannian  foliations of the sphere up to leave dimension 3. 
Moreover,  due to \cite{Wilking},  a Riemannian submersion of the round $\Sph^{15}$ 
with $7$-dimensional fibers is metrically congruent to the Hopf fibration.
Combining this work with Theorem~\ref{cor: main} gives 

\begin{cor}\label{thm: metric}
Let $\mathcal F$ be a Riemannian foliation on a round sphere $\Sph^n$ with leaf dimension $0<k<n$.
Then, up to isometric congruence, either $\mathcal F$ is given by the orbits of an isometric action of $\R$ 
or $\gS^3$ with discrete isotropy groups  
or it is the Hopf fibration of $\Sph^{15}\to \Sph^8(1/2)$ with fiber $\Sph^7$.
\end{cor}

As has been pointed out by Gromoll and Grove a real representation $\rho\colon  \gS^3\rightarrow \SO(n)$ induces 
an almost free action of $\gS^3$ on the unit sphere if and only if all irreducible subrepresentations are even dimensional.

 The paper is structured as follows. In Section \ref{sec: top} we recall 
the results stated after \tref{thm: main} and study the fibration $\hat M \to \hat B$
from a homotopy $n$-sphere $\hat M$ to the resolution $\hat B$ of the orbifold $B=M/\mathcal F$.  The fiber  of the fibration is  $\mathcal L$,  the principal leaf of $\mathcal F$, and  we  only need to consider the cases $\mathcal L =  \Sph ^7$ and $\mathcal L=  \RP ^7$. From this fibration we compute the cohomology of $\hat B$.
These computations show  $n=15$.  Moreover, the cohomology
of $\hat B$ at all primes but $2$ is concentrated in dimensions $8k,8k-1$.
In the two subsequent sections, we exclude the possibility that the orbifold $B$ is not
a manifold. Here we  use the local data of the orbfold  to find non-trivial cohomology classes of $\hat B$ that cannot exist by the previous computations.
We rely on the fact that all isotropy groups of $B$ act  freely on a $7$-dimensional sphere or a projective space, a  sever  restriction on the possible group structure.
In Section \ref{sec: even}, we use the computation of the cohomology of $\hat B$ at the prime $2$, to deduce that all isotropy groups are cyclic of odd order.   Here we  detect the forbidden classes by loooking at single points of $B$, i.e., by finding the  non-trivial restrictions of the cohomology classes to the classifying spaces of the isotropy groups.
In Section \ref{sec: p odd},   we exclude the possibility  that the set $B_p$  of points with 
non-trivial $p$-isotropy is non-empty, otherwise detecting forbidden cohomology classes by their non-trivial restriction to a component of $B_p$.

\section{Topology}\label{sec: top}
\subsection{Recollection}
Let $(M, \mathcal F)$ be as in \tref{cor: main} and assume that the leaves have dimension $k\ge 2$.
Due to \cite{ghys}, all leaves 
of $\mathcal F$ are closed. This in turn is equivalent to saying that $\mathcal F$ is  a  \emph{generalized  Seifert fibration}
 on $M$, i.e., the space of leaves $B=M/\mathcal F$ carries the natural structure of a smooth Riemannian 
 orbifold, such that the induced Riemannian distance corresponds to the distance between leaves in $M$.
 Due to  \cite{ghys},  the regular leaf $\mathcal L$ of $\mathcal F$ is a rational homology sphere.
 Following Haefliger, we consider the $\SO(n-k)$ bundle $FrM$ over $M$ given by all oriented horizontal 
 frames in $M$. Then  the Riemannian foliation $\mathcal F$ induces a fiber bundle structure on 
 $FrM$ with the fibers being diffeomorphic to $\mathcal L$ and with the base space 
 being the oriented frame bundle  $Fr B$ of the orbifold $B$.  
 Furthermore, the natural fiber bundle map  $Fr M\rightarrow FrB$ is $\SO(n-k)$-equivariant. 
 
 Thus one also gets a fiber bundle with total space given by 
  $\hM= Fr M\times_{\SO(n-k)}E\SO(n-k)$ with fiber $\mathcal L$ and  with base space 
   $\hB:=FrB\times_{\SO(n-k)}E\SO(n-k)$, $f\colon \hM\rightarrow \hB$.
  Clearly $\hM$ is homotopically equivalent to $M$ and $\hB$ is the so called 
  resolution (or classifying space) of the orbifold $B$. 
  Its cohomology is the so called orbifold cohomology of $B$. 
  As has been oberserved by Haefliger the natural projection $\hB\to B$ is  a rational homotopy equivalence.

Since the fiber $\mathcal L$ is a $k$-dimensional manifold and $\hM\sim_{heq} M\sim_{heq} \Sph^n$
 is $k$-connected, we see that 
the fiber  is contractible in $\hM$. Therefore $\mathcal L$ is an $H$-space (\cite{spanier}).
Since $\mathcal L$ is a rational homology sphere, we may apply  \cite{Browder}
and deduce that $\mathcal L$ is homotopy equivalent to $\RP^3$, $\Sph^3$, $\Sph^7$ or $\RP ^7$.
An application of the geometrization conjecture  proves \tref{cor: main}  for $k=3$.
Thus we only need to consider  the case $k=7$.

Thus $\mathcal L$ is  either homeomorphic to   $\Sph^7$  or it is homotopy equivalent to
 $\RP ^7$ and its double-cover is homeomorphic to $\Sph^7$.   
 We call the first case the \emph{spherical case} and the second case the \emph{projective case}.

\subsection{Gysin sequence and dimension} Let $R$ be any ring with unit. In the projective case 
we assume in addition that $2$ is invertible in $R$, e.g., 
$\fF_3$ or $\fQ$.
 Then $H^{\ast} (\folL, R ) = H^{\ast} (\Sph^7, R )$.
 Thus we find  the Gysin sequence of the fibration $f$ with coefficients in $R$.
 The Euler class  $a\in H^8(\hat B,R)\cong H^0(\hat B,R)\cong R$ is a generator.
 Moreover,  the cup product  $\cup a : H^{2i} (\hat B) \to H^{2i+8} (\hat B)$
 is an isomorphism, if  $2i \neq n -7$.

Since $\hB$ has finite rational cohomology, we use this isomorphism for $R=\fQ$ 
to see that $n=8l +7$, for some positive integer $l$.

\subsection{Reduction to $n=15$.}
We want to show $l=1$. 
Assume  on the contrary $l\ge 2$.  Then, due to the above isomorphism, we have $H^{\ast} (\hat B, \fF _3)=
\fF_3 [a]$ in degrees $\leq 16$.  To obtain a contradiction, we first show:

\begin{lem}
Under the assumptions above  there exists a space $X$ and an element 
$c\in H^8 (X,\fF _3)$  such that the cohomology ring  $H^{\ast} (X,\fF_3)$  equals
the polynomial ring $\fF_3 [c]$  in degrees $\leq 24$.
\end{lem}

\begin{proof} For $l >2$, one could just take $X=B$. In general,  
let $E_f$ be the mapping cylinder of $f$, which is a fiber bundle over $\hat B$ 
with fiber being the cone over $\mathcal L$.  Let $X$ be the \emph{Thom space} of
the fibration $f$, which is obtained from $E_f$ by identifying all points on the boundary of $E_f$.
For the subbundle $E' =\hat B$ of the bundle $E_f \to \hat B$, we can apply  \cite{Hatcher},
Theorem 4.D.8.  Using  the fact that the bundle $\hat M \to \hat B$ is orientable,
we deduce that there is an element
$c\in  H^8 (E, E',\fF_3) = H^8 (X, \fF_3) $ (the \emph{Thom class} of the fibration), such that 
$b\to f^{\ast} (b) \cup c$ induces an isomorphism between  $H^{\ast} (\hat B)$ and the 
reduced cohomology $\tilde H^{\ast +8} (X ,\fF_3 )$. 

The claim follows from this isomorphism and the structure of $H^{\ast} (\hat B)$.
 \end{proof}

We now get a contradiction to the following
application of Steenrod powers,  cf. \cite{Hatcher}, Theorem  4.L.9.


%


\begin{lem}  \label{Steenrod}
Let $X$ be a topological space. Assume that $H^{12} (X, \mathbb F_3) = H^{20} (X, \mathbb F_3) =0$. 
Then there is no element $c\in H^8 (X, \mathbb F_3)$, with $c^3 \neq 0$.  
\end{lem}

\begin{proof}
Consider the Steenrod operations 
$P^i\colon H^n (X, \mathbb F_3 ) \to H^{n + 4i} (X, \mathbb F_3)$.  We have 
$c^3 =P^4 (c)$.
  On the other hand, by the Adem relations, $P^4(c) $ is a  linear combination
  of $P^1(P^3(c))$ and $P^3(P^1(c))$, which must be zero, since the corresponding cohomology groups are trivial.
  Thus $c^3=0$.
\end{proof}

{

The  contradiction shows $l=1$, hence $n=15$. 
 Thus $B$ has dimension $8$ and $\hat B$ has the  rational homology of $\Sph^8$.

\subsection{Cohomology of $\hat B$} From the Gysin sequence
of the fibration $f:\hat M \to \hat B$ we deduce:

\begin{lem}\label{lem: cohomology}\label{cohom}
Let $p$ be a prime number, with $p\neq 2$ in the projective case.
Then either $\hat B$ is an $\mathbb F_p$-homology sphere, or the 
$\mathbb F_p$-cohomology ring of $\hat B$ has the form
$$H^{\ast } (\hat B, R) = R [a,b]/ b^2 =0,$$
where $b$  has degree $15$ and $a$ has degree $8$.
\end{lem}

 We will need:

\begin{lem}\label{4integer} $H^4(\hat B,\Z)=0$.
\end{lem}

\begin{proof} In the spherical case  $\hB$ is $7$-connected.
In the projective case,
we know $\pi_2 (\hat B) = \mathbb Z_2$ and $\pi_k (\hat B)=0$
for $k=1$ and $3\leq k\leq 7$.   Hence the canonical map from $\hat B$ to the 
Eilenberg-MacLane space $K(\mathbb Z _2, 2)$ 
induces isomorphisms on all cohomologies in all   degrees $\leq 7$.   
The result follows from the computations of the cohomology groups of $K(\mathbb Z_ 2 ,2)$ (for instance, cf. \cite{Clement}). 
\end{proof}

The last result about the cohomology of $\hat B$ which we  extract from the fibre bundle 
$\hat M\to \hat B$  is the following  application of the transgression theorem of
Borel (\cite{Borel}, Theorem 13.1).  
 This theorem applies (cf.  \cite{Browder},
last paragraph on p. 370),
since   in the projective case, the  fiber
$\mathcal L$ has the cohomology of $\mathbb R \mathbb P^7$.  

\begin{lem}  \label{Borel}
Assume that $\mathcal L$ is homotopy equivalent to $\mathbb R \mathbb P ^7$.
Then the cohomology ring $H^{\ast} (\hat B ,\mathbb F_2)$ up to degree $14$ is freely generated by
elements $u_2,u_3,u_5$ of degree $2,3$ and $5$ respectively.
In particular, we have  $\dim   H^{10} (\hat B, \mathbb F_2)=4$ and $\dim H^{14} (\hat B ,\mathbb F_2) =6$.
\end{lem}

\section{Isotropy groups are cyclic groups of odd order } \label{sec: even}

In this section we use  characteristic classes to see that any $2$-Sylow subgroup 
of any isotropy group is cyclic of order at most $4$. 
With different methods we then use this to show that all isotropy groups 
are cyclic groups of odd order.

Consider $B$ as the quotient space $B= FrB /\SO(8)$, where $FrB$ is the bundle of
oriented frames of $B$ with canonical action of $\SO(8)$.  Recall that the 
space $\hat B$ is nothing else but the Borel construction $\hat B = FB \times_{\SO(8)} E\SO(8)$.
We will often consider the canonical $8$-dimensional vector bundle (the \emph{tangent bundle} of the orbifold)
\[
T\hB:=FB \times_{\SO(8)} E\SO(8)\times \R^8
\] over $\hB$.

\begin{lem}\label{lem: p1}\label{lem: stiefel}
Let $V$ be a vector bundle over $\hB$. Then the Stiefel-Whitney classes
$w_2 (V)$ and $w_4 (V)$ vanish.
\end{lem}

\begin{proof} We first  assume $w_2 (V) =0$ and prove that this implies $w_4 (V)=0$.

By stabilizing with a trivial bundle 
we may assume that the rank $l$ of $V$ is at least $5$.
Let $\pr\colon \hB\rightarrow B\SO(l)$ be the classifying map of the bundle $V$.
In particular, the Stiefel-Whitney classes of $V$ are given by pull backs of Stiefel-Whitney classes of the universal bundle over $B\SO(l)$. 
Since $w_2(V)=0$,  $pr$ can be lifted to a map $\tilde pr \colon \hB\rightarrow B\Spin(l)$.
Suppose now on the contrary that $w_4(V)\neq 0$. 
Then 
\[\tilde pr^*\colon H^4(B\Spin(l),\fF_2)\to H^4(\hB,\fF_2) \] is not zero.
Since $H^4(\Spin(l),\Z)\cong \Z$ there is a natural map $B\Spin(l)\rightarrow K(\Z,4)$ to the Eilenberg-MacLane space 
$K(\Z,4)$ that induces an isomorphism on $4$-th cohomology with integral coefficients. Since this map is $5$-connected 
it also induces an isomorphism on $4$-th cohomology with $\fF_2$-coefficients.
By composing the map with $\tilde{pr}$ we get a map $\hB\to K(\Z,4)$ which induces a nontrivial map 
on $4$-th cohomology with $\fF_2$-coeffcients. On the other hand, the homotopy classes of maps
$\hB\to K(\Z,4)$ are classified by $H^4(\hB,\Z)=0$ (see Lemma~\ref{4integer}) and thus any map $\hB\to K(\Z,4)$
is null homotopic -- a contradiction.

Assume now $w_2(V)\neq 0$.   Then  $w_2 (V) ^2 \neq 0 $ as well
(cf. \lref{Borel}).  Consider the bundle $W=V\oplus V$.
Then the total Stiefel-Whitney classes satisfy $w_* (W) = w_* (V) \cdot w_* (V)$.
Since $\hB$ is simply connected,  $w_1 (V)=0$.  We deduce $w_2 (W)=0$ and 
$w_4 (W)= w_2 (V) ^2$.  Applying the previous observation to the bundle $W$, we deduce $w_4 (W)=0$.
This contradicts $w_2 (V) ^2 \neq 0$.


\end{proof}

\begin{lem}\label{lem: 2sylow}\label{lem: 2}\label{sylow2}\label{lem: sylow2} Let  $\Gamma_x\subset \SO(8)$ be
an isotropy group. 
Then any element of order $2$ is given by $-\id\in \SO(8)$. The $2$-Sylow group of $\Gamma_x$ is 
a cylic group of order at most $ 4$.
\end{lem}

\begin{proof} Let $\tilde x\in FrB$ be a point in the inverse image of $x\in B$ such that $\Gamma_x$ is the isotropy group 
of the $\SO(8)$-action on $FrB$ at $\tilde x$. 
Notice that the image of 
$\SO(8) \star  \tilde x \times  E\SO(8)\subset FrB\times E\SO(8)$ 
  under the natural projection 
$FrB\times E\SO(8)\to \hB$ can be naturally identified with the classifying space $B\Gamma_x\subset \hB$
of the isotropy group $\Gamma _x$. 
If we restrict the canonical bundle $T\hB$ over $\hB$ to $B\Gamma_x$, we get an $\R^8$-bundle 
which is isomorphic to $E\Gamma_x\times_{\Gamma_x} \R^8$ where $\Gamma_x\subset \SO(8)$ is acting by the canonical 
representation on $\R^8$. 
Let $\Gamma_0\subset \Gamma_x$ be a subgroup. If we pull back $T\hB$
via the covering map $B\Gamma_0\rightarrow B\Gamma_x\hookrightarrow B$, we thus get a bundle 
which is isomorphic to  $V= E\Gamma_0\times_{\Gamma_0} \R^8$ over $B\Gamma_0$.
By Lemma~\ref{lem: stiefel}, the second and the fourth Stiefel-Whitney classes of $V$ vanish.

 
 Suppose now that $\Gamma_0\cong \Z_2$ 
and suppose the nonzero element $\iota\in \Gamma_0\subset \SO(8)$ has $2k$ times the eigenvalue 
$-1$. Then $E\Gamma_0\times_{\Gamma_0} \R^8$ is a bundle over $\RP^\infty\cong B\Gamma_0$ which decomposes as the  sum of $2k$
canonical line bundles and $8-2k$ trivial line bundles.  
Thus the total Stiefel-Whitney class is given by $(1+w)^{2k}=(1+w^2)^k$, 
where $1$ 
is the 
generator of $H^0(\RP^\infty,\fF_2)$ and $w$ is the generator of   $H^1(\RP^\infty,\fF_2)\cong \fF_2$.

If $k$ is odd, we get $w_2 (V)\neq 0$ and if $k=2$, we see that
$w_4(V) \neq 0$.  Since $w_2 (V)=0$ and $w_4 (V)=0$ we obtain a contradiction.
This only leaves us with the possibility that $\iota$ has $2k=8$ times the eigenvalue $-1$ and thus $\iota=-\id$.



Since there is at most one order $2$ element, it follows that 
a $2$-Sylow subgroup $\gS_2\subset \Gamma_x$ does not contain any abelian noncyclic subgroup. 
This implies that  $\gS_2$ is either cyclic or generalized quaternionic 
(\cite{wolf}, \cite{Wall}). 
In order to prove that $\gS_2$ is cyclic it suffices to rule out the possibility  
that we can realize the quaternion group with $8$ elements $\gQ_8$ 
as a subgroup of an isotropy group $\Gamma_x\subset \SO(8)$.
Suppose on the contrary we can.  
As before, the bundle $V_8=E\gQ_8\times_{\gQ_8}\R^8$  over $B\gQ_8$ 
can be seen as a pull back bundle of the canonical bundle over $\hB$. 
By Lemma~\ref{4integer}, $H^4(\hB,\Z)=0$ and thus the first 
Pontryagin class of $V_8$ vanishes, $p_1(V_8)=0$.

The embedding of $\gQ_8\subset \SO(8)$ is determined by the fact 
that the center of $\gQ_8$ is mapped to $\pm \id$. 
The representation of $\gQ_8$ 
decomposes into two equivalent $4$-dimensional subrepresentations of $\gQ_8$.
Thus $V_8$ is isomorphic to the sum of two copies of the 
$4$-dimensional bundle $V_4=E\gQ_8\times_{\gQ_8}\R^4$, where $\gQ_8$ acts by its unique $4$-dimensional 
irreducible representation on $\R^4$.   
Since $V_4$ admits a complex structure, we have $c_1(V_4\otimes_\R\C)=0$ 
and thus the first Pointryagin  class is 
additive: $2p_1(V_4)=p_1(V_8)=0$. In other words,  $p_1(V_4)\in \Z_2\subset \Z_8\cong  H^4(B\gQ_8,\Z)$. 
If we pull back the bundle $V_4$ to $B\Z_4$ via the natural covering $B\Z_4\rightarrow B\gQ_8$ 
we get a bundle $V_4^*$ which decomposes
 into two $2$-dimensional subbundles, whose Euler classes 
are generators of $H^2(B\Z_4,\Z)\cong \Z_4$. 
This in turn implies that $p_1(V_4^*)$ is given by the order two element in $H^4(B\Z_4,\Z)\cong \Z_4$. 
On the other hand $p_1(V_4^*)$ is given by the image of $p_1(V_4)$ under the natural 
homomorphism $H^4(B\gQ_8,\Z)\cong\Z_8\to \Z_4\cong H^4(B\Z_4,\Z)$ -- a contradiction 
since any homomorphism $\Z_8\to \Z_4$ has $p_1(V_4)\in \Z_2\subset \Z_8$ in its kernel.

Thus the $2$-Sylow  group is cyclic. It remains to rule out that there are elements of order $8$.
Suppose on the contrary that 
$\Gamma_0\subset \Gamma_x\subset \SO(8)$ is cyclic group of order $8$ and fix a generator $\gamma\in \Gamma_0\subset \SO(8)$ .
Let $\zeta\in\gS^1\subset \C$ be a primitive 
$8$th root   and choose numbers $m_1,\ldots,m_4\in \Z$ such that 
$\zeta ^{\pm m_i}\in \gS^1\subset \C$  ($i=1,\ldots,4$) are the eigenvalues of $\gamma\in \SO(8)$ 
counted with multiplicity. Since we know $\gamma^4=-\id$, all $m_i$ are odd.

The bundle $W_8=E\Gamma_0\times_{\Gamma_0} \R^8$ over $B\Gamma_0$ decomposes into 
four orientable $2$-dimensional subbundles whose Euler classes are given
by $\pm m_i \eta$ ($i=1,\ldots,4$) where $\eta$ is a generator of $H^2(B\Gamma_0,\Z)\cong \Z_8$.

It follows that the first Pontryagin class of the bundle is given by $-(\sum_{i=1}^4m_i^2)\eta^2$.
As before $p_1(W_8)=0$ and since $\eta^2$ is a generator of $H^{4}(B\Gamma_0,\Z)=\Z_{8}$, this implies 
$m_1^2+m_2^2+m_3^2+m_4^2\equiv 0 \mod  8$. 
But for any odd number we have $m_i^2\equiv 1 \mod 8$ -- a contradiction.
 
\end{proof}

\begin{lem} \label{lem: structure} 
Any  isotropy group is either cyclic or isomorphic to a 
semidirect product $\Z_q\rtimes \Z_{4}$, where 
$\Z_4$ acts on the cyclic group of odd order $q$
by an automorphism of order $2$.
Moreover, if $\Gamma$ has even order it has a nontrivial $4$-periodic  $\fF_2$-cohomology.
\end{lem}

\begin{proof}  Let $\Gamma$ be a (not necessarily proper) subgroup of an isotropy group.
Recall that $\Gamma$ or a $\Z_2$-extension of $\Gamma$ acts freely on $\Sph^7$ and thus 
has $8$-periodic cohomology (cf. \cite{Wall}, \cite{wolf} for this fact and subsequent results about groups 
with periodic cohomology).
Thus for all odd $p$, the $p$-Sylow groups are cyclic. 
By Lemma~\ref{sylow2}, the $2$-Sylow group is cyclic as well.

A classical theorem of Burnside implies that such a group is metacyclic, that is, $\Gamma$ 
is isomorphic to a semidirect product $\Z_q\rtimes \Z_r$ where
$q$ and $r$ are relatively prime.

It remains to check that the homomorphism $\beta\colon \Z_r\rightarrow \mathrm{Aut}(\Z_q)$ 
does not contain any elements of odd prime order $p$.
In fact then Lemma~\ref{sylow2} implies that the image of $\beta$ has order at most $2$.

We argue by contradiction and assume that $\Gamma\cong  \Z_q\rtimes \Z_r$ is a minimal counterexample. 
The minimality easily implies that $q$ is a prime and that 
 $r$ is a prime power $r=p^k$, where $p\neq q$ are both odd.

We consider  the normal covering $B\Z_q\rightarrow B\Gamma$ whose deck transformation group
is generated by an element $\iota$ of order $p^k$. 
Since the order is prime to $q$, the induced map 
$H^*(B\Gamma,\fF_q)\rightarrow H^*(B\Z_q,\fF_q)$ is injective and its image is given by the 
fixed point set of $\iota^*$ where $\iota^*$ is the induced map on cohomology.
 Clearly $\iota^*$ acts on $H^2(B\Z_{q},\fF_{q})$ by an element of order $p$. 
 This in turn implies that $H^{2k}(B\Z_{q},\fF_{q})$ is fixed by $\iota^*$ if and only 
 $k$ is divisible by $p$. Hence the minimal period of $H^*(\Gamma,\Z)$ is divisible by $2p$ -- a contradiction 
 since we know that $\Gamma$ has $8$-periodic cohomology.
 Thus $\Gamma$ is cyclic and has $2$-periodic cohomology,
 or $\Gamma\cong  \Z_q\rtimes \Z_4$, where $\Z_4$ acts by an automorphism $\iota$ 
 of order two on $\Z_q$. 
 To see that in the latter case $\Gamma$ has $4$-periodic cohomology we 
 construct a free linear action of $\Gamma$ on $\Sph^3$.
 Let $\Z_m\subset \Z_q$ be the fixed point set of $\iota$. 
 Since $\iota$ has order $2$, the numbers $m$ and $q/m$ are relatively prime. 
 In particular $\Gamma\cong \Z_m\times (\Z_{q/m}\rtimes \Z_4)$. 
 We can now embed $\Gamma$ into $\mathsf{U}(2)$ by mapping the factor $\Z_m$ injectively to a central 
 subgroup of $\mathsf{U}(2)$ and by mapping $(\Z_{q/m}\rtimes \Z_4)$ injectively to 
 a subgroup of $\mathsf{SU}(2)$. Clearly, the induced action on $\Sph^3$ is free and thus $\Gamma$
 has $4$-periodic cohomology. The $\fF_2$-cohomology of $\Gamma$ can not be trivial 
 as $H^1(B\Gamma,\fF_2)\cong \fF_2$.
 \end{proof}

\begin{lem}\label{lem: cyclic} The isotropy groups are cyclic groups of odd order.
\end{lem}

\begin{proof} By Lemma~\ref{lem: structure} it suffices to show the isotropy groups 
have odd order. 
By Lemma~\ref{lem: 2} the subset $B_2\subset B$ of points whose isotropy groups 
have even order is finite  $B_2=\{p_1,\ldots,p_h\}$. Let  
$\Gamma_1,\ldots, \Gamma_h$ denote the isotropy groups.
Suppose on the contrary that $B_2$ is not empty. 

Let $Fr B_2$ denote the inverse image of $B_2$ in the frame bundle $FrB\rightarrow B$ and
 $\hat B_2= Fr B_2\times_{\SO(8)}E\SO(8)$ the corresponding subset 
in the Borel construction $\hat B= FrB\times_{\SO(8)}E \SO(8)$. 
By assumption there is a tubular neighbourhood $U$ of $\hat B_2$ in $\hat B$ 
which is homeomorphic to the normal bundle of $\hat B_2$ in $\hat B$. 
By excision and  Thom isomorphism the relative cohomology group 
$H^*(\hB, \hB\setminus \hB_2,\fF_2)$ is given by 
$\oplus_{j=1}^h H^{*-8}(B\Gamma_j,\fF_2)$. 
Furthermore, the $\fF_2$-cohomology of $\hB\setminus \hB_2$ coincides with the $\fF_2$-cohomology of 
$B\setminus B_2$ and thus is zero in degrees above $8$.  
Since $\Gamma_i$ has nontrivial $4$-periodic  $\fF_2$-cohomology we can 
combine all this with the exact sequence of the relative cohomology of the pair 
$(\hB,\hB\setminus \hB_2)$ to see that 
$\hB$ has nontrival $4$-periodic $\fF_2$-cohomology in all degrees $\ge 9$.

In the spherical case we get a contradiction to Lemma~\ref{lem: cohomology}. 
In the projective case this contradicts Lemma~\ref{Borel}.
\end{proof}

\begin{rem} Once one has established that any order two element in an isotropy group 
is given by $-\id$, one can also proceed differently to rule out isotropy groups 
of even order altogether: As above, 
there are only finitely many points $x_i\in B$ whose isotropy groups $\Gamma_i, i=1,\ldots,h$ have even order.
Moreover,  the $2$-Sylow group of $\Gamma_i$ is either cyclic or generalized quaternionic. 
By a theorem of Swan \cite{Swan} this implies that the $\fF_2$-cohomology of $\Gamma_i$ is 
nontrivial and $4$-periodic.
One can then directly pass to the proof of Lemma~\ref{lem: cyclic}. 
\end{rem}

%
%

%

\section{All isotropy groups are trivial }\label{sec: p odd}

We have seen in the last section that all isotropy groups are cyclic groups of odd order, Lemma~\ref{lem: cyclic}. 
We fix an odd prime $p$. 
 In this section we plan to prove that the order of any isotropy group 
 is not divisible by $p$. We argue by contradiction and assume the set $B_p$ of points in $B$ whose 
 isotropy group has $p$-torsion is not empty. 

 In any isotropy group $\Gamma_x$ with $x\in B_p$
there is a unique normal subgroup of $\Gamma_x$ which is isomorphic to $\Z_p$.
This implies that $B_p$ is a smooth suborbifold of $B$. 
 Let $X$ denote a 
component  of   $B_p$. 

Let $Fr X$ denote the inverse image of $X$ in the frame bundle $FrB\rightarrow B$ and
 $\hat X= Fr X\times_{\SO(8)}E\SO(8)$ the corresponding subset 
in the Borel construction $\hat B= FrB\times_{\SO(8)}E \SO(8)$. 
By assumption there is a tubular neighbourhood $U$ of $\hat X$ in $\hat B$ 
which is homeomorphic to the normal bundle of $\hat X$ in $\hat B$.

\begin{lem}\label{lem: relative} The image $H^*(\hat B,\fF_p)\rightarrow H^*(\hat X,\fF_p)$ 
contains the kernel of $H^*(\hat X,\fF_p)\rightarrow H^*(\nu^1\hat X,\fF_p)$, where 
$\nu^1\hat X$ denotes the unit normal bundle of $\hat X$ in $\hat B$. 
If the normal bundle is orientable and $e\in H^*(\hat X,\fF_p)$ denotes its Euler class 
then the kernel of the latter map is given by the image 
of $H^*(\hat X,\fF_p)\rightarrow H^*(\nu^1\hat X,\fF_p)$, $x\mapsto x\cup e$.
\end{lem}

\begin{proof} Consider the Mayer 
Vietoris sequence of $\hat B= U\cup (\hat B \setminus \hat X)$
\[
H^*(\hat B)\stackrel{j}{\rightarrow} H^*( U )\oplus H^*(\hat B\setminus  \hat X)\rightarrow H^*(U\setminus \hat X). 
\]
Since $U$ is homotopy equivalent to $\hat X$ and $ H^*(U\setminus \hat X)$ is homotopy equivalent 
to $\nu^1(\hat X)$ the first statement follows. 
The second statement is an immediate consequence of the exactness of the Gysin sequence.
\end{proof}

We will use that the cohomology $H^l(B\Z_{p},\Z)$ is given by $0$ for all odd $l$ and by 
$\Z_{p}$ for all even positive $l$. It is generated by elements in degree $0$ and $2$. 
Furthermore $H^*(B\Z_{p},\fF_p)\cong \fF_p[x,y]/x^2\fF_p[x,y]$ where $x$ has degree $1$  and $y$ has degree $2$.

We distinguish among  three cases. 

\subsection{Case 1. The normal bundle of $\hat X$ is orientable.} 
Let $x\in X$ be a point
and let $B\Gamma_x\subset \hX$ be the fiber of $x$ with respect to the natural projection $\hat B\rightarrow B$.

Then there is a unique normal subgroup  
$\Z_{p}\subset \Gamma_x$ and there are natural maps $B\Z_{p}\rightarrow B\Gamma_x\rightarrow \hat X$. 
Consider the induced map $\alpha^{\ast}\colon H^{\ast}(\hat X,\fF_p)\rightarrow H^{\ast}(B\Z_{p},\fF_p)$. 
The Euler class $e\in H^t(\hat X,\fF_p)$  of the normal bundle of $\hX\subset \hB$ is mapped to  
the Euler class $\alpha^*e$ of the bundle $E\Z_{p}\times_\rho \nu_x(\hat B)$, where $\rho\colon \Z_{p}\rightarrow O(\nu_x(\hat B))$ 
denotes the natural representation. 
The representation $\rho $ decomposes into $2$-dimensional irreducible subrepresentation and, by construction, 
each of these $2$-dimensional subrepresentations is effective. 
This in turn implies that the Euler class $\alpha^*e$ of the bundle is a generator of $H^t(B\Z_{p},\fF_p)$, where $t$ is the codimension of $X$. 
 Hence $(\alpha^*e)^k$ 
 is not zero. By Lemma~\ref{lem: relative}, this non-zero element lies in the image of $H^*(\hat B,\Z)\rightarrow  H^*(B\Z_{p},\fF_p)$. 
We deduce that $H^{kt}(\hB,\fF_p)$ does not vanish for all $k\in \fN$. 
Combining with Lemma~\ref{lem: cohomology}  this  gives  $t=8$. 
Thus $X$ is a single point and $\hat X=B\Gamma_x$.

Since $\Gamma_x$ is cyclic we have $H^{l}(B\Gamma _x,\fF_p)\cong \fF_p$ for all $l\ge 0$.
Finally, since cupping with the Euler class induces an isomorphism, we can use Lemma~\ref{lem: relative} once more to see that 
$H^{l}(\hat B,\fF_p)\neq 0$ for all $l\ge 8$ -- this contradicts  Lemma~\ref{lem: cohomology}. 



\subsection{Case 2. $\dim(X)\neq 4$ and the normal bundle of $\hat X$ is not orientable.}
Since $B$ is an orientable orbifold this 
assumption  
 implies that $X$ is a nonorientable orbifold and, in particular, 
$X$ is not a point. Thus $\dim(X)\in \{2,6\}$.
 
We consider the twofold cover $\tilde X\rightarrow \hat X$ such that the pull back 
of the normal bundle is orientable.  The map $H^*(\hat X,\fF_p)\rightarrow H^*( \tilde X,\fF_p)$ is injective 
and its image
is given by the fixed point set of $\iota^*$, where $\iota^*$ is the map induced by the nontrivial 
deck transformation  $\iota$ of $\tX$.
 
 By assumption, we know that the Euler class $e$ of the pull back bundle satisfies $\iota^* e=-e$. 
 As before we deduce that the image of $e$ in $H^{*}(B \Z_{p},\fF_p)$ does not vanish. 
 Therefore, $e^l \in H^{lt}( \tilde X,\fF_p) $ is a nontrivial element 
 in the kernel of $  H^{lt}( \tilde X,\fF_p)\rightarrow H^{lt}( \nu^1(\tilde X),\fF_p)$ for $l\ge 1$. 
 If $l$ is even $e^{l}$ is the pull back of an element  $f^{l/2} \in H^{lt} (\hX,\fF_p)$, with $f\in  H^{2t}( \tX,\fF_p)$. 
 Clearly, $f^{l/2}$ is in kernel of $ H^{lt}( \hat X,\fF_p)\rightarrow H^{lt}( \nu^1(\hat X),\fF_p)$
 and, by Lemma~\ref{lem: relative}, $H^{lt}(\hat B,\fF_p)\neq 0$ for all even 
$l$.  
 Since $t=8-\dim(X)\in \{2,6\}$, this is a contradiction to Lemma~\ref{lem: cohomology}.

 \subsection{Case 3. $\dim(X)=4$ and the normal bundle of $\hat X$ is not orientable.} This case is technically more 
 involved and we subdivide its discussion into several steps.\\[1ex]
{\bf  Step 1.} Each normal space $\nu_y(\hat X)$ of a point $y\in \hat{X}$ decomposes into 
two inequivalent $2$-dimensional subrepresentations of $\Z_{p}\subset \Gamma_y$.\\[1ex]
It is clear that $\nu_y(\hat X)$ decomposes into two subrepresentations of $\Z_{p}\subset \Gamma_y$.
If the two representations would be equivalent 
then each element $g\in \Z_p$ would naturally induce a complex structure $J$ on the normal space, 
and up to the sign the complex structure would not depend on the choice of $g$.  Since $\pm J$ induce the same 
orientation this would imply that $\nu(\hat X)$ is orientable -- a contradiction.\\[1ex]

Again, instead of working directly with $\hat X$ we go to a  suitable cover $\tilde X$.  
This time we consider   a fourfold cover  in which
the pull-back of the bundle $\nu $ is orientable and decomposes into the sum of two
orientable $2$-dimensional subbundles  determined by the first step  above.    
We summarize the properties of this cover $\tilde X$, which are intuitively rather clear,
but whose exact derivation requires some tedious considerations:\\[1ex]
%
{\bf Step 2.} There is a normal cover $\tilde X $ of $\hat X$   whose group of deck transformation is generated
by one element $\iota$ of order $4$, such that the following holds true:
\begin{enumerate}
\item The pull-back bundle $\nu (\tilde X)$ of $\nu$ to $\tilde X$  is orientable
and sum of two orientable  $2$-dimensional subbundles.  The map $\iota$ exchanges the subbundles and the map 
$\iota^2 $ changes the orientation of each of them.

\item The unit bundle $\nu ^1 (\tilde X)$ has vanishing cohomology in degrees
 $\geq 8$ with coefficients in $\fF_p$.

\item $\tilde X$ is the total space of a fiber bundle $\tilde X\to \tilde Y$ with
fiber $B\Z_p$  and connected structure group.

\item The restrictions of both  $2$-dimensional  subbundles  of $\nu (\tX)$ to a fiber $B\Z _p$ have 
 Euler classes which generate $H^2 (B\Z_p ,\Z)$.
\end{enumerate}
Moreover, $p\equiv 1 \mod 4$.
%
%
\begin{proof}  
As before  $FrX\subset FrB$ denotes the inverse image of $X$ in the frame bundle 
of $B$. Let $x\in X$ be a point, with isotropy group $\Gamma _x \subset \SO(8)$. 
Let 
$\Gamma $ be the unique normal subgroup of $\Gamma _x$ isomorphic to $\Z _p$.

We have seen above that $\Gamma$ acts on $\R ^8$ as the sum of two inequivalent 
representations and a trivial four-dimensional representation.  Therefore, the
normalizer $\gN$ of $\Gamma$  which is contained in $\Or(4) \times \Or(4) \cap \SO(8)$ has  connected component 
$\gN^0 =\SO(4) \times \gT^2 $.  Moreover,
$\gN^0$ coincides with the centralizer of $\Gamma$.  We see that $\gN$ has either two or four connected components.

Let $L\subset  Fr X$ be a  fixed point component of $\Gamma$, whose projection to $X$ is surjective. 
Then $L$ is $\gN^0$-invariant.   If $L$ is not $\gN$-invariant,
or if $\gN$ has only two connected components, then we could make 
a continuous choice of  pairs $\{g,g^{-1}\}\in \Gamma $ along $L$.
We can then argue, similarly to the first paragraph of Case 3, that the normal bundle of $\hat X$ is orientable,
in contradiction to the assumption.  

We deduce that $\gN/\gN^0$ has $4$ elements. Thus $\gN$ is isomorphic to  $\SO(4)\rtimes (\gT^2\rtimes \Z_4)$ 
where $\Z_4$ acts effectively on $\gT^2$ and $\gT^2\rtimes \Z_4$  
acts on $\SO(4)$ as $\Z_2$.   Moreover, $\gN$ acts on $\Gamma$ as the group $\Z _4$.
In particular, $p\equiv 1$ mod $4$ because otherwise $\Aut(\Z_p)$ does not contain elements of order $4$.

The generator $\iota$ of this group $\Z _4$ exchanges the $2$-dimensional 
$\Gamma$-invariant subspaces of $\R ^4 \subset \R ^8$. The square  
 $\iota ^2$ preserves the subspaces and changes the orientation on each of them.

Since all isotropy groups of points in $L$ with respect to the $\SO(8)$-action on $FrX$
are contained in $\gN$ and the $\SO(8)$-orbit through any point of $Fr X$ intersects $L$,
we see that $FrX$ is $\SO(8)$-equivariantly diffeomorphic to $L\times_{\gN}\SO(8)$. 
This in turn shows that $\hX= FrX\times_{\SO(8)}E\SO(8)$ 
is homotopy equivalent to $L\times_{\gN}E\gN$.

We now consider the $4$-fold cyclic cover $\tX = L\times _{\gN^0} E\gN$ of $\hX$
with the group of deck transformations  $\gN/\gN^0 =\Z _4$. 
Note that  the normal bundle  $\nu (L)$ decomposes as a sum of $\gN^0$-invariant 
orientable $2$-dimensional subbundles.  Hence,  the bundle $\nu (L) \times _{\gN^0} E\gN$  decomposes as a sum of
orientable $2$-dimensional subbundles.  But this bundle is just the pull-back to $\tX$ of the normal bundle of $\hX$.

  The description the action  of $\iota$ on $\R^4$ above finishes the proof
of  (1).

The unit bundle $\nu ^1 (\tX)$  is a covering of the unit bundle $\nu ^1 (\hX)$.
The latter space is homotopy equivalent to the  resolution of a $7$-dimensional 
orbifold without $p$-isotropy.  This implies (2).

In order to see (3), observe that $\Gamma =\Z_p$ lies in the kernel of the action
of $\gN$ on $L$.  Thus we have a canonical action of $\gN/\Gamma$  (which is isomorphic to $\gN$)  on $L$. 
Consider now the canonical action of 
$\gN$ on $E\gN$ and via $\gN/\Gamma$ on $E(\gN/\Gamma )$.  Then, for the diagonal action of $\gN$ on $L\times E\gN \times
E(\gN/\Gamma )$, we see that $\tilde X$
is homotop to $L\times _{\gN^0}  (E\gN\times E (\gN/\Gamma ))$.  The canonical 
projection of this space  to $\tY := L\times _{\gN^0} E(\gN/\Gamma )$
is a fiber bundle with fiber  $B\Gamma$.  Moreover, the  structure group of this bundle is the
connected group $\gN^0$.

The restriction of each of the  $2$-dimensional subbundles to the fiber $B\Z_p$ 
is given by $E\Z_p\times_{(\Z_p,\rho_i)}\R^2$ where $\rho_1$ and $\rho_2$ are the two 
inequivalent faithful representations mentioned at the beginning. 
This proves (4).
\end{proof}

The last statement, namely $p\equiv 1 \mod 4$, implies
that  any endomorphism 
of order $4$ on any finite-dimensional $\fF_p$-vector space
is diagonalizable  with eigenvalues $\lambda\in \fF_p$ satisfying 
$\lambda^4=1\in \fF_p$.   In particular, it is true for the endomorphism 
 $\iota^*$ of $H^*(\tX,\fF_p)$.

  If   $e$ denotes the Euler class (with any coefficients) of the  bundle $\nu (\tX)$
and $e_i$ denote the Euler classes of the two $2$-dimensional subbbundles, then
the first  statement of the above lemma reads as follows:  $e_1 \cup e_2 = e$;
$\iota ^*$ preserves the set of four elements $\{\pm e_1,\pm e_2\}$;
and $(\iota ^* )^2(e_i)=-e_i$, for $i=1,2$.\\[1ex]  
{\bf Step 3.} Let $\mathcal I ^{\ast}$ denote the graded subalgebra of 
$H^{\ast} (\tX, \fF_p)$ that consists of  $\iota ^{\ast}$-invariant elements
divisible by the Euler class $e$ of $\nu (\tX)$. 
 Then $\dim (\mathcal I ^8 ) =1$ and $\mathcal I ^k =0$ for $0<k< 15, k\neq 8$.\\[1ex]
The natural map  $H^{\ast} (\hX, \fF_p )\rightarrow H^{\ast} (\tX, \fF_p )$ 
is injective and as in Case 2 its image is given by  the $\iota^{\ast}$-invariant elements. 
The subalgebra $\mathcal I ^{\ast}$ is thus isomorphic to the kernel of 
$H^{\ast} (\hX, \fF_p )\rightarrow H^{\ast} (\nu^1(\hX), \fF_p )$. 
Combining  Lemma~\ref{lem: relative}
and \lref{lem: cohomology} Step 3 follows.\\[1ex]
%
%
%
{\bf Step 4.} $H^1(\tX,\fF_p)= 0$.\\[1ex]
Otherwise,  choose a non-zero
eigenvector $w\in H^1(\tX,\fF_p)$ of $\iota^*$.
%
In the subspace $H^2(\tX,\fF_p)$ spanned by $e_1$ and $e_2$ we can find an
eigenvector  $f$ of $\iota^*$ which is not in kernel of the restriction to $H^2( B\Z_{p},\fF_p)$. 
Of course, the Euler class $e$ satisfies $\iota^* e =-e$.
Since $f^2$ restricts to a generator of $H^4 ( B\Z_{p},\fF_p)$,  we see that
  $\iota^*f=h f$ with $h^2\equiv -1$ mod $p$.

We claim that $w\cup f^l\cup e\neq 0$ for all $l\ge 0$. 
We choose a circle $\Sph^1\subset \tY$ in the base (see Step 2 (3)) such that $w$ restricts 
to a nonzero element in the first $\fF_p$ cohomology group of the inverse image 
$\tS$ of $\Sph^1$ in $\tX$. We get a  fiber bundle $B\Z_p\to \tS\to \Sph^1$ 
and since the structure group is connected this bundle must be trivial.
Since $f$ and $e$ restrict to nonzero elements 
of the $\fF_p$-cohomology of the fiber $B\Z_p$ the claim follows.

Depending on the eigenvalue of $w$, we  can choose
 some $l\in \{0,1,2,3\}$ such that $w\cup f^l \cup e$ is fixed by $\iota^*$.
The existence of this  non-zero element of $\mathcal I ^k$ with $k \in {5,7,9,11}$
contradicts Step 3. \\[1ex]
%
%
%
%
%
%
%
%
%
{\bf Step 5.} For all $j>0$, we have $\dim( H^{2j}(\tX,\fF_p))\ge 2$.
\begin{proof}
By the previous step $H^1(\tX,\fF_p)= 0$.  The  group
$H_1(\tX,\Z)$ is  finite without $p$-torsion, thus $H^2(\tX,\Z)$ does not have $p$-torsion either. 

Let $R$ be the ring obtained by localizing $\Z$ at $p$, i.e. 
\[
 R=\Z\bigl[\{1/q\mid \, \mbox{$q$ is a prime with } q\neq p\}\bigr]\subset \fQ
\]
From the universal coefficient theorem 
$H^1(\tX,R)=0$ and $H^2(\tX,R)=R^r$, for some $r$.  
Let $\he_1,\he_2\in H^2(\tX,R)$ denote the Euler classes  with $R$ coefficients 
of the two  $2$-dimensional subbundles of $\nu (\tX)$.    
 Due to Step 2 
 they restrict to generators of 
$H^2(B\Z_p,R)\cong \Z_p$. In particular $\he_i\neq 0$. 
Moreover $(\iota^*)^2 \he_i=-\he_i$. Thus $\iota^*$ acts as an endomorphism 
of order four  on  $H^2(\tX,R)=R^r$. Therefore $r\ge 2$.


We consider the fibration $B\Z_p\rightarrow \tX\rightarrow \tY$. 
Clearly $H^2(\tY,R)$ has rank at least $2$ as well.

We look at the Serre spectral sequence with coefficients in 
$R$.  Since  the action of the fundamental group on the cohomology of the fiber is trivial,  
the $E_2$ page is given by $E_2^{i,j}=H^i(\tY,H^j(B\Z _p,R))$.  
The $0$-th column $E_2^{0,j}$ survives throughout the sequence since $H^*(\tX,R)\to H^*(B\Z_p,R)$ is surjective. 
Therefore also the $0$-th entry $E_2^{2,0}$ of the second column survives  
 throughout. 
Notice that the odd entries in second column are zero while the even positive 
entries of the $E_2$-page are all isomorphic to $H^2(\tY,\Z_{p})$. For each of these even dimensional 
entries the natural image of $H^2(\tY,R)\to H^2(\tY,\Z_{p})$  
coincides with the image of  $E_2^{0,2j}\otimes E_2^{2,0}$ in $E^{2,2j}_2$ with respect to the multiplicative structure
since the multiplicative structure is induced by the cup product. 
Clearly these subgroups survive till the $E_\infty$ term.
Notice that the image of $H^2(\tY,R)$ in $H^2(\tY,\Z_p)$ is given by $(\Z_{p})^r$ 
for some $r\ge 2$. 
Therefore $H^{2k}(\tX,R)$ is the domain of a surjective homomorphism to $(\Z_p)^2$ for all positive $k$.
\end{proof}
%
%
%
%
%
%
%
%
%
%
%
%
%
%
%
A contradiction in Case 3 now arises as follows.
Since $\nu^1(\tX)$ can be seen as a resolution of a $7$-dimensional orbifold 
whose isotropy groups do not have $p$-torsion it follows that $H^i(\nu^1(\tX),\fF_p)=0$ for all $i\ge 8$.  
We see 
from the Gysin sequence for $\nu^1(\tX)$  that cupping with $e$ induces an isomorphism of the cohomology groups 
in degrees $\ge 5$.  
 Since $e=e_1\cup e_2$ the same holds for cupping with $e_1$. 
Moreover,  cupping with $e$  is surjective onto $H^8 (\tX , \fF_p )$.  

By Step 5 we can choose an eigenvector
$w\in H^8(\tX,\fF_p)$, which is linear independent to the fixed point $e^2$.


If $\iota^*w=w$, then   $\dim (\mathcal I ^8 ) \geq 2$.
If $\iota^*w=-w$, then $w\cup e\in H^{12}(\tX,\fF_p)$ would be a nonzero 
element of $\mathcal I ^{12}$. In both cases we get a contradiction to Step 3.  


Otherwise we have that $(\iota^*)^2 w=-w$. Then $w\cup e_1$ is a nonzero fixed point of $(\iota^*)^2$.
This in turn implies that $H^{10}(\tX,\fF_p)$ contains an eigenvector of $\iota^*$ to the eigenvalue 
of $1$ or $-1$. In the latter case cupping with $e$ gives a nontrivial element of
$\mathcal I ^{14}$.  In  the former case we have a nonzero element in $\mathcal I ^{10}$,
providing a contradiction to Step 3  in both cases.


\section{Final Remarks.} \label{sec: 2 sph}

In summary we have ruled out all orbifold singularties in $B$. Thus $B$ is a Riemannian manifold, 
and $\mathcal F$ is given by the fibers of a Riemannian submersion $M\rightarrow B$. 
 By \cite{Browder}  (or  Lemma~\ref{Borel}),  it  follows that we are in the spherical case.
From the homotopy sequence of the fiber bundle, we see that the base $B$ of the submersion is a homotopy sphere,
hence $B$ is a topological sphere. This finishes the proof of Theorem~\ref{thm: main}.

\begin{rem} It is well known (\cite{Shimada}) that there are many exotic $15$-spheres that fiber over $\Sph^8$. 
Of course, the base manifold $B$ in part c) of the main theorem can also be an exotic sphere, in fact
one can just pull back the Hopf fibration to the exotic $8$-sphere by a smooth degree $1$ map from the exotic $8$-sphere 
to $\Sph^8$.
What is not known however is whether the fibers of such a fibration can be exotic 
$7$-spheres. This seems to be closely related to the question 
on how closely the diffeomorphism group of an exotic $7$-sphere is linked to the diffeomorphism 
group of $\Sph^7$.
\end{rem}

\bibliographystyle{alpha}
\bibliography{sphfibration}

\begin{thebibliography}{Ghy84}

\bibitem[Bor53]{Borel}
A.~Borel.
\newblock Sur la cohomologie des espaces fibr\'es principaux et des espaces
  homog\`enes de groupes de lie compacts.
\newblock {\em Ann. of Math.}, 57:115--207, 1953.

\bibitem[Bro62]{Browder}
Browder.
\newblock Higher torsion in $h$-spaces.
\newblock {\em Ann. Math.}, 116:70--97, 1962.

\bibitem[Cle02]{Clement}
A.~Clement.
\newblock {\em Integral cohomology of finite Postnikov towers}.
\newblock PhD thesis, Lausanne, 2002.

\bibitem[GG88]{gromollsph}
D.~Gromoll and K.~Grove.
\newblock The low-dimensional metric foliations of euclidean spheres.
\newblock {\em J. Differential Geom.}, 28:143--156, 1988.

\bibitem[Ghy84]{ghys}
E.~Ghys.
\newblock Feuilletages riemanniens sur les vari\'et\'es simplement connexes.
\newblock {\em Annales de l'institut Fourier}, 34:203--223, 1984.

\bibitem[GZ00]{groveziller}
K.~Grove and W.~Ziller.
\newblock Curvature and symmetry of milnor spheres.
\newblock {\em Ann. of Math.}, 152:331--367, 2000.

\bibitem[Hae84]{classify}
A.~Haefliger.
\newblock Groupoids d'holonomie et classifiants.
\newblock {\em Asterisque}, 116:70--97, 1984.

\bibitem[Hat02]{Hatcher}
A.~Hatcher.
\newblock {\em Algebraic topology}.
\newblock Cambridge university press, 2002.

\bibitem[Oli79]{Oliver}
R.~Oliver.
\newblock Weight systems for $so(3)$ actions.
\newblock {\em Ann. of Math.}, 110:227--241, 1979.

\bibitem[Shi57]{Shimada}
N.~Shimada.
\newblock Differentiable structure on the 15 sphere and pontrjagin classes of
  certain manifolds.
\newblock {\em Nagoya Math. J.}, 12:59--69, 1957.

\bibitem[SW54]{spanier}
E.H. Spanier and Whitehead.
\newblock On fiber spaces in which the fiber is contractible.
\newblock {\em Ann.}, 116:70--97, 1954.

\bibitem[Swa60]{Swan}
R.~Swan.
\newblock A $p$-period of a finite group.
\newblock {\em Ill. J. Math.}, 4:341--346, 1960.

\bibitem[Wal10]{Wall}
C.~Wall.
\newblock On the structure of groups with periodic cohomology.
\newblock Preprint, 2010.

\bibitem[Wil01]{Wilking}
B.~Wilking.
\newblock Index parity of closed geodesics and rigidity of hopf fibrations.
\newblock {\em Invent. Math.}, 148:281--295, 2001.

\bibitem[Wol67]{wolf}
J.~Wolf.
\newblock {\em Spaces of constant curvature}.
\newblock McGraw-Hill, 1967.

\end{thebibliography}

\end{document}